\documentclass[reqno, 12pt]{amsart}
\usepackage{amsmath}
\usepackage{amssymb}
\usepackage{amsthm}
\usepackage{enumerate}
\usepackage{upgreek}
\usepackage[mathscr]{eucal}
\usepackage{graphicx}

\newtheorem{theorem}{Theorem}[section]
\newtheorem{lemma}[theorem]{Lemma}
\newtheorem{proposition}[theorem]{Proposition}
\newtheorem{corollary}[theorem]{Corollary}
\newtheorem{definition}[theorem]{Definition}
\newtheorem{definitions}[theorem]{Definitions}
\theoremstyle{definition}
\newtheorem{example}[theorem]{Example}
\newtheorem{remark}[theorem]{Remark}

\begin{document}
\title[]{Cardinal Functions, Bornologies and Strong Whitney convergence}

\author{Tarun Kumar Chauhan \and Varun Jindal*}

\thanks{*Supported by NBHM Research Grant 02011/6/2020/NBHM(R.P) R\&D II/6277.}

\address{Tarun Kumar Chauhan: Department of Mathematics, Malaviya National Institute of Technology Jaipur, Jaipur-302017, Rajasthan, India}
\email{rajputtarun.chauhan@gmail.com}

\address{Varun Jindal: Department of Mathematics, Malaviya National Institute of Technology Jaipur, Jaipur-302017, Rajasthan, India}
\email{vjindal.maths@mnit.ac.in}



\begin{abstract} 
Let $C(X)$ be the set of all real valued continuous functions on a metric space $(X,d)$. Caserta introduced the topology of strong Whitney convergence on bornology for $C(X)$ in \cite{SWc}, which is a generalization of the topology of strong uniform convergence on bornology introduced by Beer-Levi in \cite{Suc}. The purpose of this paper is to study various cardinal invariants of the function space $C(X)$ endowed with the topologies of strong Whitney  and Whitney convergence on bornology. In the process, we present simpler proofs of a number of results from the literature. In the end, relationships between cardinal invariants of strong Whitney convergence and strong uniform convergence on $C(X)$ have also been studied. 

\end{abstract}	

\maketitle

\section{Introduction}
A family $\mathcal{B}$ of nonempty subsets of a set $X$ is called a $\textit{bornology}$ on $X$ if $\mathcal{B}$ forms an ideal and covers $X$. A subfamily $\mathcal{B}_0$ of $\mathcal{B}$ satisfying that for every $B\in\mathcal{B}$ there is an element $B'\in\mathcal{B}_0$ such that $B\subseteq B'$ is called \textit{base} for the bornology $\mathcal{B}$. If $\mathcal{B}$ is a bornology on a metric space $(X,d)$ with base $\mathcal{B}_0$ such that every member of $\mathcal{B}_0$ is closed (or compact) in $(X,d)$, then $\mathcal{B}$ is said to have a \textit{closed} (or \textit{compact}) \textit{base}.  

The smallest (respectively, largest) bornology on $X$ is the collection $\mathcal{F}$ of all finite (respectively, $\mathcal{P}_0(X)$ of all nonempty) subsets of $X$. Another important bornology on $X$ is $\mathcal{K}$, the family of all nonempty \textit{relatively compact} subsets (that is, subsets with compact closure) of $X$. For more details about bornologies on metric spaces see \cite{Ballf,Tbab,Bc}.



For a metric space $(X,d)$ and a bornology $\mathcal{B}$ on $X$, the most commonly used topology on $C(X)$ is the classical topology of uniform convergence on $\mathcal{B}$, denoted by $\tau_{\mathcal{B}}$.  In \cite{Suc}, Beer and Levi introduced a stronger version of the topology $\tau_{\mathcal{B}}$, the topology of strong uniform convergence on $\mathcal{B}$ and denoted by $\tau_{\mathcal{B}}^s$. The study of the topology $\tau_{\mathcal{B}}^s$ was further continued in \cite{Ucucas,ATasucob,Cmotosucob,Cfbafs}. In \cite{SWc}, Caserta generalized the topology $\tau_{\mathcal{B}}^s$, and introduced  a new topology called as the topology of strong Whitney convergence on $\mathcal{B}$, denoted by $\tau_{\mathcal{B}}^{sw}$. The topology $\tau_{\mathcal{B}}^{sw}$ is a stronger form of the classical function space topology $\tau_{\mathcal{B}}^{w}$ of Whitney convergence on $\mathcal{B}$. Recently, the topologies  $\tau_{\mathcal{B}}^{sw}$ and $\tau_{\mathcal{B}}^{w}$ have been studied in \cite{SWcob} and \cite{SWasucob}. The topology $\tau_{\mathcal{B}}^{w}$ is a generalization of the well-known Whitney topology $\tau^w$. It reduces to $\tau^w$ when $\mathcal{B} = \mathcal{P}_0(X)$. The Whitney topology $\tau^w$ was introduced by H. Whitney in \cite{whitney}, and further studied in \cite{Hewitt,Bpifs,hola_novo2,DHHM,M1}.  

In this paper, for a metric space $(X,d)$ and a bornology $\mathcal{B}$ on $X$, we study some important cardinal functions on $C(X)$ endowed with the topologies $\tau_{\mathcal{B}}^{sw}$ and $\tau_{\mathcal{B}}^{w}$. More precisely, we consider character, tightness, weight and network weight. Since these spaces are topological groups, we also consider their index of narrowness. In order to study these cardinal functions, we first examine two new cardinal invariants associated with bornology named as (strong) domination number of $X$ with respect to $\mathcal{B}$. These two cardinal invariants play a pivotal role in all our investigations. Using the index of narrowness of the space $(C(X),\tau_{\mathcal{B}}^{w})$, we present simpler proofs of the results proved in \cite{DHHM} (see, Remark \ref{General remark}). Besides this, to make this study more interesting, we investigate the relationships of these cardinal invariants of $(C(X),\tau_{\mathcal{B}}^{sw})$ with the corresponding cardinal invariants of $(C(X),\tau_{\mathcal{B}}^{s})$. In particular, we show that the character and tightness of $(C(X),\tau_{\mathcal{B}}^{sw})$ may be strictly greater than the character and tightness of $(C(X),\tau_{\mathcal{B}}^{s})$. However, the density, weight, network weight, and index of narrowness of both the spaces are equal.

\section{Preliminaries}

All metric spaces are assumed to have at least two points. For any $x\in X$ and $\delta>0$, $S_{\delta}(x)$ denotes the open ball with center $x$ and radius $\delta$. For any nonempty subset $A$ of $X$,  $A^{\delta}$ represents the $\delta$-\textit{enlargement} of $A$ defined as $A^{\delta} = \cup_{x\in A}S_{\delta}(x)$. Note that $A^{\delta} = \overline{A}^{\delta}$, where $\overline{A}$ denotes the closure of $A$. We denote by $f_0$ the constant function zero on $X$, that is, $f_0: X\to \mathbb{R}$ such that $f_0(x) = 0$ for all $x \in X$. For other terms and notations, we refer to \cite{engelking,Tposocf}. 

Let $\mathcal{B}$ be a bornology on a metric space $(X,d)$. Then the classical \textit{topology $\tau_{\mathcal{B}}$ of uniform convergence on $\mathcal{B}$} for $C(X)$ is determined by the uniformity which has a base for its entourages all sets of the form
\begin{align*}
[B,\epsilon] = \{(f,g) : |f(x)-g(x)| < \epsilon \text{ for all } x\in B\} \ \ \ (B \in \mathcal{B}, \epsilon > 0).
\end{align*}
and the classical uniformity for the \textit{topology $\tau_{\mathcal{B}}^w$ of Whitney convergence on $\mathcal{B}$} has as a base for its entourages all sets of the form
\begin{align*}
[B,\epsilon]^w = \{(f,g) : |f(x)-g(x)| < \epsilon(x) \text{ for all }x\in B\} \ \ \ (B\in \mathcal{B}, \epsilon\in C^+(X)).
\end{align*}
\noindent Here $C^+(X)$ represents the set of all positive real-valued continuous functions defined on $X$. Note that if $\mathcal{B} = \mathcal{P}_0(X)$, then the above uniformity generates the topology $\tau^w$ (Whitney topology) of Whitney convergence on $C(X)$.


The \textit{topology $\tau_\mathcal{B}^s$ of strong uniform convergence on $\mathcal{B}$} is determined by a uniformity on $C(X)$ having as a base all sets of the form
\begin{align*}
[B,\epsilon]^s = \{(f,g): \ \exists~\delta>0 \ \forall x\in B^\delta, |f(x)-g(x)|< \epsilon\} \ \ (B \in \mathcal{B}, \epsilon > 0).
\end{align*}
and the \textit{topology $\tau_{\mathcal{B}}^{sw}$ of strong Whitney convergence on $\mathcal{B}$} is determined by a uniformity having as a base all sets of the form
\begin{align*}
[B,\epsilon]^{sw} = \{(f,g) : \exists ~ \delta > 0 \ \forall x\in B^\delta,  |f(x)-g(x)| < \epsilon(x)\}
\end{align*}
with $B\in \mathcal{B}, \epsilon\in C^+(X)$.

\noindent In general, on $C(X)$ the above defined topologies are related as follows:
\begin{align*}
\tau_{\mathcal{B}} \subseteq \tau_{\mathcal{B}}^{s} \subseteq \tau_{\mathcal{B}}^{sw} ~ \text{ and } ~ \tau_{\mathcal{B}} \subseteq \tau_{\mathcal{B}}^{w} \subseteq \tau_{\mathcal{B}}^{sw} \subseteq \tau^w.
\end{align*}
The relationships between the above mentioned topologies are thoroughly studied in \cite{SWasucob}.

The concept of a shield introduced by Beer et al. in \cite{Bcas} plays an important role in the study of the topologies $\tau_{\mathcal{B}}^{s}$ and $\tau_{\mathcal{B}}^{sw}$ (see, \cite{Ucucas} and \cite{SWasucob}). Recall that for a nonempty subset $A$  of $X$, a superset $A_1$ of $A$ is called a \textit{shield} for $A$ provided that for every closed subset $C$ of $X$ with $C\cap A_1 = \emptyset$, we have $C\cap A^\delta = \emptyset$ for some $\delta > 0$. Since $A^{\delta} = \overline{A}^{\delta}$, a set $A_1$ is a shield for $A$ if and only if $A_1$ is a shield for $\overline{A}$. Evidently, $X$ is a shield for all $A\in \mathcal{P}_0(X)$.

A bornology $\mathcal{B}$ on $X$ is called \textit{shielded from closed sets} if $\mathcal{B}$ contains a shield for each of its members. It is known that every bornology with a compact base is shielded from closed sets. 
We now recall the definitions of various cardinal invariants considered in this paper for an arbitrary topological space $(X,\tau)$.

\hspace{-0.6cm}$\begin{array}{lrl}
\text{character of } x	\text{ in } X & \chi(X,x) &= \aleph_0 + \min\{|\mathscr{B}_x| : \mathscr{B}_x \text{ is a base for } X \text{ at } x\}. \\ 
\text{character of } X	&  \chi(X) &= \sup\{\chi(X,x) : x\in X\}.\\ 
\text{tightness of } x \text{ in } X  & t(x,X) &= \aleph_0 + \min\{\mathfrak{m}:\text{for every } C\subseteq X \text{ with }  \\
& & \ \ \  x\in\overline{C}, \text{ there exits a } C_0\subseteq C  \text{ with } \\
& &  \ \ \  |C_0|\leq\mathfrak{m} \text{ such that } x\in\overline{C_0}\}.\\
\text{tightness of } X   & t(X) &= \sup\{t(x,X) : x\in X\}. \\
\text{density of } X    & d(X) & =  \aleph_0 + \min\{|D| : D \text{ is a dense subset of } X\}. \\
\text{weight of } X  & w(X) & =  \aleph_0 + \min\{|\mathscr{B}| : \mathscr{B} \text{ is a base for } X\}. \\
\text{Lindel\"{o}f degree of } X  & L(X) & =  \aleph_0 + \min\{\mathfrak{m} : \text{ every open cover of } X \\
& & \ \ \  \text{has a subcover with cardinality}\leq\mathfrak{m}\}. \\
\text{cellularity of } X   & c(X) & =   \aleph_0 + \sup\{|\mathscr{U}| : \mathscr{U}\subseteq \tau \text{ and } \mathscr{U} \text{ is a pairwise} \\
 & & \ \ \  \text{disjoint family}\}.\\
 \text{network weight of } X &\ \ nw(X) & =  \aleph_0 + \min\{|\mathscr{N}|: \mathscr{N} \text{ is a network on } X\}.
 
\end{array} $

A collection $\mathscr{N}$ of nonempty subsets of $X$ is called a \textit{network} on $X$ if for each $x\in X$ and $U\in \tau$ with $x\in U$, there exists a $C\in\mathscr{N}$ such that $x\in C \subseteq U$. 


In a special case when $(X,\tau)$ is a topological group, another important cardinal function may be defined on $(X,\tau)$. This cardinal invariant is called the index of narrowness (\cite{tkachenko}). If $(X,\tau)$ is a topological group under addition, then for an infinite cardinal number $\lambda$, $X$ is called $\lambda$-\textit{narrow} if for every neighborhood $U$ of the identity element in $X$, there exists a subset $S$ of $X$ with $|S|\leq \lambda$ such that $X=U+S=\{u+s:u\in U \text{ and } s\in S\}$. Then
 
\vspace{0.2cm}
\hspace{-0.55cm}$\begin{array}{lcl}
\text{index of narrowness of } X & \ \ \ ib(X) & =  \aleph_0 + \min\{\lambda: G \text{ is } \lambda\text{-narrow}\}.
\end{array} $
\vspace{0.2cm}

 \section{Character of $(C(X),\tau_{\mathcal{B}}^{sw})$ and $(C(X),\tau_{\mathcal{B}}^{w})$}

In order to study the character of the spaces $(C(X),\tau_{\mathcal{B}}^{sw})$ and $(C(X),\tau_{\mathcal{B}}^{w})$, we first investigate two new cardinal functions on $X$ named as domination number and strong domination number of $X$ with respect to a family $\mathcal{B}$ of subsets of $X$. 



\begin{definition}\normalfont{Let $\mathcal{B}$ be a family of nonempty subsets of a topological space $X$. A subset $F_{\mathcal{B}}$ of $C(X)$ is called a \textit{dominating subset of $C(X)$ with respect to $\mathcal{B}$} if for every $f \in C(X)$ and every $B\in \mathcal{B}$ there exists $g \in F_{\mathcal{B}}$ such that $f(x) \leq g(x)$ for all $x\in B$. Then the \textit{domination number of $X$ with respect to $\mathcal{B}$} is defined by
		\begin{align*}
		dn_{\mathcal{B}}(X) = \aleph_0 + \min\{&|F_{\mathcal{B}}| : F_{\mathcal{B}} \text{ is a dominating subset of } C(X) \\ &\text{ with respect to } \mathcal{B}\}.
		\end{align*}}	 \end{definition}
When $\mathcal{B} = \mathcal{P}_0(X)$, then $dn_{\mathcal{B}}(X) = dn_{\mathcal{P}_0(X)}(X)$ is simply denoted by $dn(X)$ and is known as \textit{domination number of $X$} (see \cite{mkatetov,DHHM}), and a dominating subset of $C(X)$ with respect to $\mathcal{B}$ for $\mathcal{B}= \mathcal{P}_0(X)$ is known as a \textit{dominating subset of $C(X)$}.

\begin{definition}\normalfont{Let $(X,d)$ be a metric space and let $\mathcal{B}$ be a family of nonempty subsets of $X$. A subset $F_{\mathcal{B}}^s$ of $C(X)$ is called a \textit{strongly dominating subset of $C(X)$ with respect to $\mathcal{B}$} if for every $f \in C(X)$ and every $B\in \mathcal{B}$ there exist $\delta > 0$ and $g \in F_{\mathcal{B}}^s$ such that $f(x) \leq g(x)$ for all $x\in B^\delta$. Then the \textit{strong domination number of $X$ with respect to $\mathcal{B}$} is defined by
		\begin{align*}dn_{\mathcal{B}}^s(X) = \aleph_0 + \min\{&|F_{\mathcal{B}}^s| : F_{\mathcal{B}}^s \text{ is a strongly dominating subset of } C(X) \\ &\text{ with respect to } \mathcal{B}\}.\end{align*}}	 \end{definition}

\begin{remark}\label{inequality of domination numbers}
It can be observed that $dn_{\mathcal{B}}(X) \leq dn_{\mathcal{B}}^s(X) \leq dn(X)$ for any family $\mathcal{B}$ of nonempty subsets of a metric space $(X,d)$.	
\end{remark}

In order to prove our next theorem, we need the following fact that we state as a lemma.

\begin{lemma}\label{bddness of continuous function on compact set} Suppose $(X,d)$ is a metric space. Then every $f\in C(X)$ is bounded on some enlargement of every compact subset of $X$. \end{lemma}	

\begin{theorem}\label{bornology has compact base iff strong domination number is aleph} Let $\mathcal{B}$ be a bornology on a metric space $(X,d)$ with a closed base.  Then the following conditions are equivalent:
	\begin{enumerate}
		\item[(a)] $\mathcal{B}$ has a compact base;
		\item[(b)] $dn_{\mathcal{B}}^s(X) = \aleph_0$;
		\item[(c)] $dn_{\mathcal{B}}(X) = \aleph_0$.
\end{enumerate} \end{theorem}
\begin{proof} $(a)\Rightarrow(b)$. Suppose $F_{\mathcal{B}}^s = \{f_n \in C(X):f_n(x) = n ~ \forall x\in X \text{ and } n\in\mathbb{N}\}$. Then by Lemma \ref{bddness of continuous function on compact set}, $F_{\mathcal{B}}^s$ is a strongly dominating subset of $C(X)$ with respect to $\mathcal{B}$ with $|F_{\mathcal{B}}^s| = \aleph_0$.
	
	
	The implication	$(b)\Rightarrow(c)$ follows from the Remark \ref{inequality of domination numbers}.
	
	$(c)\Rightarrow(a)$. Suppose there is a $B\in\mathcal{B}$ such that $\overline{B}$ is not compact. So there exists a subset $D = \{x_n : n\in\mathbb{N}\}$ of $\overline{B}$ which is closed and discrete in $X$. Let $F = \{f_n : n\in\mathbb{N}\}$ be any countable subset of $C(X)$. By Tietze's extension theorem, we can find $g\in C(X)$ such that $g(x_n) = f_n(x_n)+1$. So no $f_n\in F_{\mathcal{B}}$ satisfies $g(x) \leq f_n(x)$ for all $x\in \overline{B}$. Therefore $F$ cannot be a dominating subset of $C(X)$ with respect to $\mathcal{B}$.
	\end{proof} 

Since $\mathcal{B} = \mathcal{P}_0(X)$ has a compact base if and only if $X$ is compact, we have the following corollary to Theorem \ref{bornology has compact base iff strong domination number is aleph}.
\begin{corollary}\label{DHHM theorem} For a metric space $(X,d)$, $dn(X) = \aleph_0$ if and only if $X$ is compact.
\end{corollary}

\begin{remark} Since $dn_{\mathcal{B}}(X)$ is defined for any topological space $X$, the equivalence $(a)\Leftrightarrow (c)$ in Theorem \ref{bornology has compact base iff strong domination number is aleph} holds for any Tychonoff space $X$ by replacing the condition $(a)$ with $(a')$: Every member of $\mathcal{B}$ is relatively pseudocompact subset of $X$ (that is, every $f\in C(X)$ is bounded on each member of $\mathcal{B}$). Therefore the Corollary \ref{DHHM theorem} also holds for any Tychonoff space $X$ provided the condition `$X$ is compact' is replaced with `$X$ is pseudocompact' (see, Proposition 2.1, \cite{DHHM}). \end{remark}

\begin{remark}\label{all domination number are equal} 
	If a bornology $\mathcal{B}$ contains $X$, then $dn_{\mathcal{B}}(X) = dn_{\mathcal{B}}^s(X) = dn(X)$. However, converse fails. For example, if $X=[0,1]$ and $\mathcal{B} = \mathcal{F}$, then by Theorem \ref{bornology has compact base iff strong domination number is aleph} and Corollary \ref{DHHM theorem}, $dn_{\mathcal{F}}(X) = dn_{\mathcal{F}}^s(X) = dn(X)$ but $\mathcal{F}$ does not contain $X$.\end{remark}

Observe that if a bornology $\mathcal{B}$ on $(X,d)$ is stable under small enlargements (that is, for every $B\in\mathcal{B}$, $B^{\delta}\in\mathcal{B}$ for some $\delta>0$), then $dn_{\mathcal{B}}(X) = dn_{\mathcal{B}}^s(X)$. In fact, in this case every dominating subset of $C(X)$ with respect to $\mathcal{B}$ becomes strongly dominating with respect to $\mathcal{B}$. But it follows from  the Remark \ref{all domination number are equal} that converse is not true. 

\noindent In our next result, we show that the equality $dn_{\mathcal{B}}(X) = dn_{\mathcal{B}}^s(X)$ holds even under a weaker assumption that $\mathcal{B}$ is shielded from closed sets. However, we don't know yet if the converse is also true, that is, does $dn_{\mathcal{B}}(X) = dn_{\mathcal{B}}^s(X)$ imply $\mathcal{B}$ is shielded from closed sets?


\begin{theorem}\label{B is shielded from closed sets, then dn_B^s=dn_B} Let $\mathcal{B}$ be a bornology on a metric space $(X,d)$ with a closed base. If $\mathcal{B}$ is shielded from closed sets, then $dn_{\mathcal{B}}(X) = dn_{\mathcal{B}}^s(X)$.
\end{theorem}
\begin{proof} The inequality $dn_{\mathcal{B}}(X) \leq dn_{\mathcal{B}}^s(X)$ is immediate. For the reverse inequality $dn_{\mathcal{B}}(X) \geq dn_{\mathcal{B}}^s(X)$, we show that if $\mathcal{B}$ is shielded from closed sets, then every dominating subset of $C(X)$ with respect to $\mathcal{B}$ is also strongly dominating with respect to $\mathcal{B}$. Let $F$ be a dominating subset of $C(X)$ with respect to $\mathcal{B}$. Consider $B\in\mathcal{B}$ and $f\in C(X)$. Without loss of generality, assume $B$ is closed. Suppose $B_1\in\mathcal{B}$ is a shield for $B$. Also $F$ being dominating with respect to $\mathcal{B}$ implies that there exists a function $g\in F$ such that $1+f(x)\leq g(x)$ for all $x\in B_1$. We claim that $f(x)\leq g(x)$ for all $x\in B^\delta$ for some $\delta>0$. Suppose for every $n\in \mathbb{N}$, there exists $x_n\in B^{1/n}\setminus B_1$ such that $f(x_n) > g(x_n)$. Clearly, $A\cap B_1 = \emptyset$, where $A=\{x_n:n\in\mathbb{N}\}$. If $A$ were closed, then $A\cap B^{1/n_0}=\emptyset$ for some $n_0\in\mathbb{N}$ as $B_1$ shields $B$. Which is impossible. Otherwise, $A$ has a cluster point $x_0 \in X$. Then $x_0\in \overline{B}=B\subseteq B_1$ so that $f(x_0) < 1+f(x_0)\leq g(x_0)$. But from $f(x_n) > g(x_n)$ for all $n \in \mathbb{N}$ and the continuity of $f$ and $g$, we also have $f(x_0) \geq g(x_0)$. We arrive at a contradiction.
	\end{proof}

\begin{definition}\normalfont{For a metric space $(X,d)$ and a bornology $\mathcal{B}$ on $X$, the \textit{weight} of $\mathcal{B}$, denoted by $w_{X}(\mathcal{B})$, is defined by
		\begin{align*}
		w_{X}(\mathcal{B}) = \aleph_0 + \inf\{|\mathcal{B}_0| : \mathcal{B}_0 \text{ is a base for } \mathcal{B}\}.
		\end{align*}
		Clearly, $w_{X}(\mathcal{P}_0(X)) = \aleph_0$. Note that $w_{X}(\mathcal{B})$ is denoted by $cf(\mathcal{B})$ in \cite{Cfbafs}. }\end{definition}



We now study the character of the spaces $(C(X),\tau_{\mathcal{B}}^{sw})$ and $(C(X),\tau_{\mathcal{B}}^{w})$. Since both of these are topological groups and therefore homogenous spaces, their character is equal to the character of $f_0$ in the respective topology on $C(X)$. 



We need the following lemma given in \cite{Fswufagt} for the next result. We include its statement here for the readers' convenience.
\begin{lemma}\label{function extension lemma}\textnormal{(Lemma 1.1 on page 7, \cite{Fswufagt})} If $A$ is a $C$-embedded subset of $X$, then any continuous function $f : A \to (0,\infty)$ can be extended to a continuous function $F : X \to (0,\infty)$.
\end{lemma} 	

\begin{theorem}\label{character of topology of SWC}Let $(X,d)$ be a metric space and let $\mathcal{B}$ be a bornology on $X$ with a closed base. Then $\chi(C(X),\tau_{\mathcal{B}}^{sw}) = dn_{\mathcal{B}}^s(X)\cdot w_{X}(\mathcal{B})$.\end{theorem}
\begin{proof} First we show that $\chi(C(X),\tau_{\mathcal{B}}^{sw}) \leq dn_{\mathcal{B}}^s(X)\cdot w_{X}(\mathcal{B})$. Let $\mathcal{B}_0$ be a base for $\mathcal{B}$ with $|\mathcal{B}_0| = w_{X}(\mathcal{B})$ and let $F_{\mathcal{B}}^s \subseteq C^+(X)$ be a strongly dominating subset of $C(X)$ with respect to $\mathcal{B}$ with $|F_{\mathcal{B}}^s| = dn_{\mathcal{B}}^s(X)$. For any $f\in C(X)$, let
	\begin{align*}
	\mathscr{B}(f) = \{[B,1/\phi]^{sw}(f) : B\in\mathcal{B}_0, \phi\in F_{\mathcal{B}}^s\}.
	\end{align*}
	\noindent Therefore $|\mathscr{B}(f)| \leq dn_{\mathcal{B}}^s(X)\cdot w_{X}(\mathcal{B})$. We show that $\mathscr{B}(f)$ is a local base at $f$ in $(C(X),\tau_{\mathcal{B}}^{sw})$. Let $[B_1,\psi]^{sw}(f)$ be any basic open neighborhood of $f$ in $(C(X),\tau_{\mathcal{B}}^{sw})$ for some $B_1\in\mathcal{B}$ and $\psi\in C^+(X)$. Then there exist a $B\in\mathcal{B}_0$ and a $\phi \in F_{\mathcal{B}}^s$ such that $B_1 \subseteq B$ and $1/\psi \leq \phi$ on $B^\delta$ for some $\delta > 0$.  Thus $[B,1/\phi]^{sw}(f) \subseteq [B_1,\psi]^{sw}(f)$. So $\mathscr{B}(f)$ is a local base at $f$ in $(C(X),\tau_{\mathcal{B}}^{sw})$. Therefore $\chi(C(X),\tau_{\mathcal{B}}^{sw}) \leq dn_{\mathcal{B}}^s(X)\cdot w_{X}(\mathcal{B})$.
	
	To show the reverse inequality $dn_{\mathcal{B}}^s(X)\cdot w_{X}(\mathcal{B}) \leq \chi(C(X),\tau_{\mathcal{B}}^{sw})$,  let $\mathscr{B}(f_0)$ be a local base at $f_0$  with $|\mathscr{B}(f_0)| = \chi(C(X),\tau_{\mathcal{B}}^{sw})$. We can assume that every member of $\mathscr{B}(f_0)$ is of the form $[B,\phi]^{sw}(f_0)$ for some closed $B\in\mathcal{B}$ and $\phi\in C^+(X)$. So let $\mathscr{B}(f_0) = \{[B_i,\phi_i]^{sw}(f_0): i \in I\}$, where $|I| = \chi(C(X),\tau_{\mathcal{B}}^{sw})$ and for each $i\in I$, $B_i\in \mathcal{B}$ is closed in $(X,d)$ and $\phi_i \in C^+(X)$. Define $\mathcal{B}_0 = \{B_i : i \in I\}$ and $F_{\mathcal{B}}^s = \{1/\phi_i : i \in I\}$. To establish the required inequality, it is sufficient to show that $\mathcal{B}_0$ is a base for $\mathcal{B}$ and $F_{\mathcal{B}}^s$ is strongly dominating with respect to $\mathcal{B}$.
	
	\noindent Suppose $\mathcal{B}_0$ is not a base for $\mathcal{B}$. Let $B'\in\mathcal{B}$ such that $B' \nsubseteq B_i$ for any $i \in I$. For every $i\in I$, choose $x_i\in B'\setminus B_i$. Since $B_i$ is closed and $x_i\notin B_i$, we have $d(x_i, B_i) > 0$ for all $i \in I$.  So for each $i \in I$, we can find $\delta_i > 0$ such that $x_i \notin \overline{(B_i)^{\delta_i}}$. Thus for each $i \in I$, by Urysohn's Lemma there is a function $g_i\in C(X)$ such that $g_i(x) = 0$ for all $x\in \overline{(B_i)^{\delta_i}}$ and $g_i(x_i) = 1$. Consequently, $g_i\in [B_i,\phi_i]^{sw}(f_0)$ for each $i \in I$ but $g_i\notin [B',1/2]^{sw}(f_0)$. Hence $[B_i,\phi_i]^{sw}(f_0) \nsubseteq [B',1/2]^{sw}(f_0)$ for any $i\in I$.  It contradicts that $\mathscr{B}(f_0)$ is a local base at $f_0$. So $\mathcal{B}_0$ forms a base for $\mathcal{B}$. Therefore $w_{X}(\mathcal{B}) \leq |I| = |\mathcal{B}_0|$.
	
	To show that $F_{\mathcal{B}}^s$ is a strongly dominating subset of $C(X)$ with respect to $\mathcal{B}$, consider $g\in C(X)$ and $B'\in\mathcal{B}$. Without loss of generality, assume $B'$ is closed. Let $\psi \in C^+(X)$ such that $g(x) \leq \psi(x)$ for all $x\in X$. Then  there exists $i \in I$ such that $[B_i,\phi_i]^{sw}(f_0) \subseteq [B',\frac{1}{2\psi}]^{sw}(f_0)$. We show that $\phi_i \leq \frac{1}{\psi}$ on $(B')^\delta$ for some $\delta>0$. So that $g\leq \psi\leq \frac{1}{\phi_i}$ on $(B')^\delta$. Suppose no such $\delta > 0$ exists, that is, for every $n\in\mathbb{N}$, there exists $x_n\in (B')^{1/n}$ such that $\frac{1}{\psi(x_n)} < \phi_i(x_n)$. Let $D=\{x_n : n\in\mathbb{N}\}$. If $D$ had a cluster point $x_0$, then $x_0\in B'$. Therefore by continuity of functions $\frac{1}{\psi}$ and $\phi_i$, we have $\frac{1}{\psi(x_0)}\leq \phi_i(x_0)$. So $\frac{1}{2\psi(x_0)} < \phi_i(x_0)$. Let $t=\frac{1}{2\psi(x_0)\phi_i(x_0)}$. Clearly, $0 < t < 1$. Thus $t\phi_i \in [B_i,\phi_i]^{sw}(f_0)$. But $t\phi_i \notin [B',\frac{1}{2\psi}]^{sw}(f_0)$ as $t\phi_i(x_0)=\frac{1}{2\psi(x_0)}$. This contradicts that $[B_i,\phi_i]^{sw}(f_0) \subseteq [B',\frac{1}{2\psi}]^{sw}(f_0)$. If $D$ had no cluster point, then $D$ is closed and discrete subset of $X$. By Lemma \ref{function extension lemma}, the continuous function $h':D\to \mathbb{R}$ such that $h'(x_n)=\frac{1}{2\psi(x_n)}$ for all $x_n\in D$ can be extended to a function $h \in C^+(X)$ such that $h|_D=h'$. Define $H = \min\{\frac{\phi_i}{2}, h\}$. Clearly, $H\in C^+(X)$ and $H(x)\leq \frac{\phi_i(x)}{2}$ for all $x\in X$. Thus $H\in [B_i,\phi_i]^{sw}(f_0)$. But $H \notin [B',\frac{1}{2\psi}]^{sw}(f_0)$ as $H(x_n) = \frac{1}{2\psi(x_n)}$ for all $n\in\mathbb{N}$. We again arrive at a contradiction. Therefore $dn_{\mathcal{B}}^s(X) \leq |F_{\mathcal{B}}^s|$.
	\end{proof}

\begin{theorem}\label{character of topology of wc}Let $(X,d)$ be a metric space and let $\mathcal{B}$ be a bornology on $X$ with a closed base. Then $\chi(C(X),\tau_{\mathcal{B}}^{w}) = dn_{\mathcal{B}}(X)\cdot w_{X}(\mathcal{B})$.\end{theorem}
\begin{proof} The inequality $\chi(C(X),\tau_{\mathcal{B}}^{w}) \leq dn_{\mathcal{B}}(X)\cdot w_X(\mathcal{B})$ can be proved in a manner similar to the proof of the corresponding inequality given in Theorem \ref{character of topology of SWC}.
	
	%
	
	We prove that $dn_{\mathcal{B}}(X)\cdot w_{X}(\mathcal{B}) \leq \chi(C(X),\tau_{\mathcal{B}}^{w})$. Suppose $\mathscr{B}(f_0)$ is a local base at $f_0$ with $|\mathscr{B}(f_0)| = \chi(C(X),\tau_{\mathcal{B}}^{w})$. Like in the proof of Theorem \ref{character of topology of SWC}, we can assume $\mathscr{B}(f_0) = \{[B_i,\phi_i]^{w}(f_0): i \in I\}$, where $|I| = \chi(C(X),\tau_{\mathcal{B}}^{w})$ and for each $i\in I$, $B_i\in \mathcal{B}$ is closed in $(X,d)$ and $\phi_i \in C^+(X)$. Define $\mathcal{B}_0 = \{B_i : i \in I\}$ and $F_{\mathcal{B}} = \{1/\phi_i : i \in I\}$. Then $\mathcal{B}_0$ is a base for $\mathcal{B}$ can be proved using similar steps as in the proof of Theorem \ref{character of topology of SWC}. We show that $F_{\mathcal{B}}$ is dominating with respect to $\mathcal{B}$.
	
	
	To show that $F_{\mathcal{B}}$ is a dominating subset of $C(X)$ with respect to $\mathcal{B}$, let us consider $g\in C(X)$ and $B'\in \mathcal{B}$. Choose $\psi\in C^+(X)$ such that $g(x) \leq \psi(x)$ for all $x\in X$. Find $i \in I$ such that $[B_i,\phi_i]^w(f_0) \subseteq [B',1/\psi]^w(f_0)$. We show that $\phi_i \leq 1/\psi$ on $B'$, so that $g\leq\psi\leq1/\phi_i$ on $B'$. Suppose it is not true, that is, $1/\phi_i(x_0) < \psi(x_0)$ for some $x_0\in B'$. Let $t = \frac{1}{\phi_i(x_0)\psi(x_0)}$. Then $t\phi_i\in[B_i,\phi_i]^w(f_0)$. However, $t\phi_i \notin[B',1/\psi]^w(f_0)$ as $t\phi_i(x_0) = 1/\psi(x_0)$. This contradiction shows that $F_{\mathcal{B}}$ is a dominating subset of $C(X)$ with respect to $\mathcal{B}$. Therefore $dn_{\mathcal{B}}(X) \leq |F_{\mathcal{B}}| \leq \chi(C(X),\tau_{\mathcal{B}}^{w})$.
	\end{proof}	

\begin{corollary}\label{DHHM result1}\textnormal{(\cite{DHHM}, Theorem 2.3)} For a metric space $(X,d)$, $\chi(C(X),\tau^{w}) =dn(X)$.
\end{corollary}



Since $dn_{\mathcal{B}}(X)\leq dn_{\mathcal{B}}^{s}(X)\leq dn(X)$ for any bornology $\mathcal{B}$ on a metric space $(X,d)$ with a closed base, we conclude that $dn_{\mathcal{B}}(X)\cdot w_{X}(\mathcal{B}) \leq \chi(C(X),\tau_{\mathcal{B}}^{sw}) \leq dn(X)\cdot w_{X}(\mathcal{B})$. However, $\chi(C(X),\tau_{\mathcal{B}}^{sw}) \neq dn(X)\cdot w_{X}(\mathcal{B})$ in general.

\begin{example}\label{counterexample1_caserta} Let $(X,d)=\mathbb{R}$ with the usual metric and $\mathcal{B}=\mathcal{K}$. Then  $w_{\mathbb{R}}(\mathcal{K})=\aleph_0$ and by Theorem \ref{bornology has compact base iff strong domination number is aleph}, $dn_{\mathcal{K}}(\mathbb{R}) = \aleph_0$. So $\chi(C(\mathbb{R}),\tau_{\mathcal{K}}^{sw})=\aleph_0$ but $dn(\mathbb{R}) > \aleph_0$.\end{example}

\noindent Example \ref{counterexample1_caserta} shows that Proposition 4 and Corollary 2 of \cite{SWc} do not hold in general. But they are true when $X$ is compact.

\begin{corollary}\label{first countability of topology of SWC on bornology}Let $(X,d)$ be a metric space and let $\mathcal{B}$ be a bornology on $X$ with a closed base. Then the following conditions are equivalent:
	\begin{enumerate}
		\item[(a)] $\chi(C(X),\tau_{\mathcal{B}}^{sw})=\aleph_0$;
		\item[(b)] $\chi(C(X),\tau_{\mathcal{B}}^{w})=\aleph_0$;
		\item[(c)] $\mathcal{B}$ has a countable base consisting of compact sets.
\end{enumerate}	\end{corollary}

\section{More Cardinal Invariants}
The main aim of this section is to study tightness and index of narrowness of  $(C(X),\tau_{\mathcal{B}}^{sw})$ and $(C(X),\tau_{\mathcal{B}}^{w})$. Besides this, we also characterize the weight and network weight of these spaces.
 

\subsection{Tightness and Network weight} \hfill \vspace{0.2cm}

In order to study the tightness of these spaces, we recall the following definitions.   

\begin{definitions}\normalfont{Let $(X,d)$ be a metric space and let $\mathcal{B}$ be a bornology on $X$. A collection $\mathcal{G}$ of open subsets of $X$ is called an \textit{open} $\mathcal{B}$-\textit{cover} of $X$ if for every $B\in\mathcal{B}$, there exists a $G\in\mathcal{G}$ such that $B\subseteq G$. Then the $\mathcal{B}$-\textit{Lindel\"{o}f degree} of $X$ is defined by
\begin{align*}L(X,\mathcal{B}) = &\aleph_0 + \min\{\mathfrak{m}: \text{every open } \mathcal{B}\text{-cover of } X \text{ has an open } \mathcal{B}\text{-subcover of }\\ &X \text{ with cardinality } \leq \mathfrak{m}\}.\end{align*}

A collection $\mathcal{G}$ of open subsets of $X$ is called a \textit{strong open} $\mathcal{B}$-\textit{cover} of $X$ if for every $B\in\mathcal{B}$, there exist a $G\in\mathcal{G}$ and a $\delta>0$ such that $B^{\delta}\subseteq G$. Then the \textit{strong} $\mathcal{B}$-\textit{Lindel\"{o}f degree} of $X$ is defined by
		\begin{align*}L^s(X,\mathcal{B}) = &\aleph_0 + \min\{\mathfrak{m}: \text{every strong  open } \mathcal{B}\text{-cover of } X \text{ has a strong open }\\ &\mathcal{B}\text{-subcover of } X \text{ with cardinality } \leq \mathfrak{m}\}.\end{align*}
	
In general, we have $L^s(X,\mathcal{B})\leq L(X,\mathcal{B}) \leq w_X(\mathcal{B})$ (Proposition 2.2, \cite{Cfbafs}). Consequently, if $\mathcal{B}$ has a countable base, then $L^s(X,\mathcal{B})= L(X,\mathcal{B})= \aleph_0$.}\end{definitions}

But $L^s(X,\mathcal{B})= L(X,\mathcal{B})$ also holds for an important class of bornologies as shown in the next result.
 		
\begin{proposition}\label{B is shielded from closed sets, then L^s(X,B)=L(X,B)} Let $\mathcal{B}$ be a bornology on a metric space $(X,d)$ with a closed base. If $\mathcal{B}$ is shielded from closed sets, then $L^s(X,\mathcal{B}) = L(X,\mathcal{B})$.
\end{proposition}
\begin{proof}  We only need to show  $L(X,\mathcal{B}) \leq L^s(X,\mathcal{B})$. Let $\mathcal{G}$ be an open  $\mathcal{B}$-cover of $X$. So for every $B\in\mathcal{B}$, there is a $G_B\in\mathcal{G}$ such that $B\subseteq G_B$.
Let $\mathcal{G}_1=\{G_{B'}\in \mathcal{G} : B'\in\mathcal{B} \text{ is a shield for some } B\in \mathcal{B}\}$. Let $B\in\mathcal{B}$ and $B'\in\mathcal{B}$ is a shield of $B$. So there is a $G_{B'}\in\mathcal{G}_1$ such that $B\subseteq B'\subseteq G_{B'}$. Consequently, $B^{\delta}\cap(X\setminus G_{B'})=\emptyset$ for some $\delta>0$, and thus $B^{\delta}\subseteq G_{B'}$. Therefore $\mathcal{G}_1$ is a strong open $\mathcal{B}$-cover of $X$. If $\mathcal{G}_0$ is a strong open $\mathcal{B}$-subcover  of $\mathcal{G}_1$ with $|\mathcal{G}_0|\leq L^s(X,\mathcal{B})$, then $\mathcal{G}_0$ is also an open $\mathcal{B}$-subcover of $\mathcal{G}$. Hence $L(X,\mathcal{B}) \leq L^s(X,\mathcal{B})$.     
\end{proof}	

The following example shows that the converse of the above proposition need not be true.

\begin{example} Let $X=\mathbb{R}$ with the usual metric. Let $\mathcal{B}$ be a bornology on $\mathbb{R}$ with a base $\mathcal{B}_0=\{[a,b]\cup\mathbb{N}: a,b \text{ are integers with } a<b\}$. Clearly, $\mathcal{B}_0$ is countable. So $L^s(X,\mathcal{B}) = L(X,\mathcal{B}) = \aleph_0$. For any $[a,b]\cup\mathbb{N} \in \mathcal{B}_0$, if we take the closed set $A=\{n+\frac{1}{n}: n\in\mathbb{N}\}\setminus [a,b]$, then $A\cap ([a,b]\cup\mathbb{N})=\emptyset$. But for each $\delta>0$, $\mathbb{N}^\delta\cap A\neq\emptyset$. So $\mathbb{N}\in\mathcal{B}$ has no shield in $\mathcal{B}$. Consequently, $\mathcal{B}$ is not shielded from closed sets.     
\end{example}	
		
	
\begin{proposition}\label{dn_B and L_B is less than tightness}Let $(X,d)$ be a metric space and let $\mathcal{B}$ be a bornology on $X$ with a closed base. Then the following inequalities hold:
	\begin{enumerate}[(i)]
		\item $dn_{\mathcal{B}}^s(X)\leq t(C(X),\tau_{\mathcal{B}}^{sw})$; and $dn_{\mathcal{B}}(X)\leq t(C(X),\tau_{\mathcal{B}}^{w})$.
		\item $L^s(X,\mathcal{B})\leq t(C(X),\tau_{\mathcal{B}}^{sw})$; and $L(X,\mathcal{B})\leq t(C(X),\tau_{\mathcal{B}}^{w})$.
	\end{enumerate} 
\end{proposition}
\begin{proof} In both parts, the proofs for $\tau_{\mathcal{B}}^{sw}$ and $\tau_{\mathcal{B}}^{w}$ are similar. So we give proof only for $\tau_{\mathcal{B}}^{sw}$.  
	
$(i)$.  It is easy to see that $f_0$ is a $\tau_{\mathcal{B}}^{sw}$-closure point of $C^+(X)$. So there is a subset $D_0$ of $C^+(X)$ with $|D_0|\leq t(f_0, \tau_{\mathcal{B}}^{sw})$ such that $f_0$ is a $\tau_{\mathcal{B}}^{sw}$-closure point of $D_0$. We now show that the set $E_{\mathcal{B}}^s=\{\frac{1}{\phi}\in C^+(X): \phi\in D_0\}$ is a strongly dominating subset of $C(X)$ with respect to $\mathcal{B}$. Let $g\in C(X)$ and $B\in\mathcal{B}$. Thus for $\psi=1+|g|$, there is a member $\phi\in D_0\cap [B,\frac{1}{\psi}]^{sw}(f_0)$. Consequently, we have $g(x)<1/\phi(x)$ for all $x\in B^\delta$ for some $\delta>0$. Therefore $dn_{\mathcal{B}}^s(X)\leq t(C(X),\tau_{\mathcal{B}}^{sw})$.
	
$(ii)$. Let $\mathcal{G}$ be a strong open $\mathcal{B}$-cover of $X$. So for each $B\in\mathcal{B}$, there exist some $G_B\in\mathcal{G}$ and $\delta_B>0$ such that $B^{\delta_B}\subseteq G_B$. For every $B\in\mathcal{B}$, we can find a function $f_B\in C(X)$ such that $f_B(x)=0$ for all $x\in \overline{B^{\delta_B/2}}$ and $f_B(x)=1$ for all $x\in X\setminus G_B$. It is easy to see that $f_0$ is a $\tau_{\mathcal{B}}^{sw}$-closure point of $E$, where $E=\{f_B : B\in \mathcal{B}\}$. So there is a subset $E_0=\{f_{B_i}: i\in I\}$ of $E$ with $|I|\leq t(f_0,\tau_{\mathcal{B}}^{sw})$ such that $f_0$ is a $\tau_{\mathcal{B}}^{sw}$-closure point of $E_0$. Take $\mathcal{G}_0=\{G_{B_i} : i\in I\}\subseteq \mathcal{G}$. Clearly, $|\mathcal{G}_0|\leq t(f_0,\tau_{\mathcal{B}}^{sw})$. So for every $B'\in\mathcal{B}$, there exists $j\in I$ such that $f_{B_j}\in E_0\cap [B',\frac{1}{2}]^{sw}(f_0)$. Thus $(B')^\delta\subseteq G_{B_j}$ for some $\delta>0$. Therefore $\mathcal{G}_0$ is a strong open $\mathcal{B}$-cover of $X$. Hence $L^s(X,\mathcal{B})\leq t(C(X),\tau_{\mathcal{B}}^{sw})$. 	
\end{proof}

\begin{proposition}\label{tightness is less than dn_B L_B} Let $(X,d)$ be a metric space and let $\mathcal{B}$ be a bornology on $X$ with a closed base. Then the following inequalities hold:
	\begin{enumerate}[(i)]
		\item $t(C(X),\tau_{\mathcal{B}}^{sw}) \leq dn_{\mathcal{B}}^s(X)\cdot L^s(X,\mathcal{B})$.
		\item $t(C(X),\tau_{\mathcal{B}}^{w}) \leq dn_{\mathcal{B}}(X)\cdot L(X,\mathcal{B})$.
	\end{enumerate} 
\end{proposition}
\begin{proof} $(i)$. Let $D\subseteq C(X)$ such that $f_0$ is a $\tau_{\mathcal{B}}^{sw}$-closure point of $D$ and let $F_{\mathcal{B}}^s \subseteq C^+(X)$ be a strongly dominating subset of $C(X)$ with respect to $\mathcal{B}$ with $|F_{\mathcal{B}}^s|=dn_{\mathcal{B}}^s$. For every $\phi\in F_{\mathcal{B}}^s$, consider the set $\mathcal{A}^{\phi}=\{(\phi \cdot f)^{-1}(-1,1): f\in D\}$. Let $B\in\mathcal{B}$. Since $f_0$ is a $\tau_{\mathcal{B}}^{sw}$-closure point of $D$, there exists a function $f_B\in D\cap [B,\frac{1}{\phi}]^{sw}(f_0)$. So we have $B^\delta \subseteq (\phi \cdot f_B)^{-1}(-1,1)$ for some $\delta>0$. Therefore $\mathcal{A}^{\phi}$ is a strong open $\mathcal{B}$-cover of $X$ for each $\phi\in F_{\mathcal{B}}^s$. So there exists a subset $\mathcal{A}^{\phi}_0=\{(\phi \cdot f_i)^{-1}(-1,1): i\in I\}$ of $\mathcal{A}^{\phi}$ with $|I|\leq L^s(X,\mathcal{B})$, which is a strong open $\mathcal{B}$-cover of $X$. Take $D_{\phi}=\{f_i\in D: i \in I\}$. Clearly, $|D_{\phi}|\leq L^s(X,\mathcal{B})$. Thus $D_0=\cup_{\phi\in F_{\mathcal{B}}^s} D_{\phi}$ is a subset of $D$ with $|D_0|\leq dn_{\mathcal{B}}^s(X)\cdot L^s(X,\mathcal{B})$. We now show that $f_0$ is a $\tau_{\mathcal{B}}^{sw}$-closure point of $D_0$. Let $[B,\epsilon]^{sw}(f_0)$ be a basic open neighborhood of $f_0$ in $(C(X),\tau_{\mathcal{B}}^{sw})$ for some $B\in\mathcal{B}$ and $\epsilon\in C^+(X)$. So there exist $\phi_B \in F_{\mathcal{B}}^s$ and $\delta_B>0$ such that $\epsilon(x)\geq \frac{1}{\phi_B(x)}$ for all $x\in B^{\delta_B}$. Since $\mathcal{A}^{\phi}_0$ is a strong open $\mathcal{B}$-cover of $X$, there exists $j\in I$ such that $B^{\delta_1}\subseteq (\phi_B \cdot f_j)^{-1}(-1,1)$ for some $\delta_1>0$. Thus $f_j\in D_{\phi_B}\cap [B,\epsilon]^{sw}(f_0)$. Therefore $t(C(X),\tau_{\mathcal{B}}^{sw}) \leq dn_{\mathcal{B}}^s(X)\cdot L^s(X,\mathcal{B})$. 
	
	$(ii)$. It can be proved in a manner similar to part $(i)$.  
\end{proof}
%

The next theorem follows from Propositions \ref{dn_B and L_B is less than tightness} and \ref{tightness is less than dn_B L_B}.

\begin{theorem}\label{tightness of topology of SWC} Let $\mathcal{B}$ be a bornology on a metric space $(X,d)$ with a closed base. Then 
	\begin{enumerate}[(i)]
		\item $t(C(X),\tau_{\mathcal{B}}^{sw}) = dn_{\mathcal{B}}^s(X)\cdot L^s(X,\mathcal{B})$.
		\item $t(C(X),\tau_{\mathcal{B}}^{w}) = dn_{\mathcal{B}}(X)\cdot L(X,\mathcal{B})$.
	\end{enumerate} 
\end{theorem}
	
Next we study the network weight of the spaces $(C(X),\tau_{\mathcal{B}}^{sw})$ and $(C(X),\tau_{\mathcal{B}}^{w})$. 
		
\begin{definition}\normalfont{Let $(X,d)$ be a metric space and let $\mathcal{B}$ be a bornology on $X$. A collection $\varLambda$ of nonempty subsets of $X$ is called a $\mathcal{B}$-\textit{network} on $X$ if for every $B\in\mathcal{B}$ and every open $U\subseteq X$ with $\overline{B}\subseteq U$, there exists a $C\in\varLambda$ such that $\overline{B}\subseteq C \subseteq U$. Then the $\mathcal{B}$-\textit{network weight} of $X$ is defined by
\begin{align*}nw(X,\mathcal{B}) = \aleph_0 + \min\{|\varLambda|: \varLambda \text{ is a } \mathcal{B}\text{-network on } X\}.\end{align*}}	 \end{definition}

\begin{proposition}\label{dn_B and nw_B is less than nw of SWC}Let $(X,d)$ be a metric space and let $\mathcal{B}$ be a bornology on $X$ with a closed base. Then the following inequalities hold:
	\begin{enumerate}[(i)]
		\item $nw(X,\mathcal{B}) \leq nw(C(X),\tau_{\mathcal{B}}^{w}) \leq nw(C(X),\tau_{\mathcal{B}}^{sw})$.
		\item $dn_{\mathcal{B}}^s(X) \leq nw(C(X),\tau_{\mathcal{B}}^{sw})$; and $dn_{\mathcal{B}}(X) \leq nw(C(X),\tau_{\mathcal{B}}^{w})$.
	\end{enumerate} 
\end{proposition}
\begin{proof} $(i)$. Since the topology $\tau_{\mathcal{B}}^{w}$ is finer than $\tau_{\mathcal{B}}$, by Theorem 3.6 of \cite{Cfbafs}, $nw(X,\mathcal{B}) \leq nw(C(X),\tau_{\mathcal{B}}^{w})$. Second inequality is immediate.
	
$(ii)$. We only give the proof of first inequality as the proof of other is similar. Let $\mathscr{N}$ be a network on $(C(X),\tau_{\mathcal{B}}^{sw})$ with $|\mathscr{N}|=nw(C(X),\tau_{\mathcal{B}}^{sw})$. For each $N\in \mathscr{N}$, pick $\phi_N \in N$ and consider the set $E=\{1+|\phi_N| : N \in \mathscr{N}\}$. Clearly, $|E|=nw(C(X),\tau_{\mathcal{B}}^{sw})$. To show $E$ is strongly dominating with respect to $\mathcal{B}$, consider $g\in C(X)$ and $B\in\mathcal{B}$. So there is a member $N\in\mathscr{N}$ such that $g\in N \subseteq [B,1]^{sw}(g)$. Then $|\phi_N|+1\in E$ and $g(x)<|\phi_N|+1$ for all $x\in B^\delta$.	Therefore $E$ is a strongly dominating subset of $C(X)$ with respect to $\mathcal{B}$. Hence $dn_{\mathcal{B}}^s(X) \leq nw(C(X),\tau_{\mathcal{B}}^{sw})$. 
\end{proof}

\begin{proposition}\label{nw of SWC is less than dn_B nw_B density of SWC}Let $(X,d)$ be a metric space and let $\mathcal{B}$ be a bornology on $X$ with a closed base. Then the following inequalities hold:
	\begin{enumerate}[(i)]
		\item $nw(C(X),\tau_{\mathcal{B}}^{sw}) \leq dn_{\mathcal{B}}^s(X)\cdot d(C(X),\tau_{\mathcal{B}}^{sw})\cdot nw(X,\mathcal{B})$.
		\item $nw(C(X),\tau_{\mathcal{B}}^{w}) \leq dn_{\mathcal{B}}(X)\cdot d(C(X),\tau_{\mathcal{B}}^{w})\cdot nw(X,\mathcal{B})$.
	\end{enumerate} 
\end{proposition}
\begin{proof} $(i)$. Let $F_{\mathcal{B}}^s$ be a strongly dominating subset of $C(X)$ with respect to $\mathcal{B}$ with $|F_{\mathcal{B}}^s|= dn_{\mathcal{B}}^s(X)$, $\varLambda$ be a $\mathcal{B}$-network on $X$ with $|\varLambda|=nw(X,\mathcal{B})$, and let $D$ be a dense subset in $(C(X),\tau_{\mathcal{B}}^{sw})$ with $|D|=d(C(X),\tau_{\mathcal{B}}^{sw})$. Define $\mathscr{A}=\{[C,\frac{1}{2\phi}]^{sw}(g) : C\in\varLambda, \ \phi\in F_{\mathcal{B}}^s \text{ and } g\in D\}$. So $|\mathscr{A}|\leq dn_{\mathcal{B}}^s(X)\cdot d(C(X),\tau_{\mathcal{B}}^{sw})\cdot nw(X,\mathcal{B})$. We show that $\mathscr{A}$ is a network on $(C(X),\tau_{\mathcal{B}}^{sw})$. Let $[B,\epsilon]^{sw}(f)$ be a basic open neighborhood of $f$ in $(C(X),\tau_{\mathcal{B}}^{sw})$. So we can find $\phi\in F_{\mathcal{B}}^s$ and $g\in D$ such that $\frac{1}{\phi(x)}\leq\epsilon(x)$ for all $x\in B^{\delta_1}$ for some $\delta_1>0$ and $g\in [B,\frac{1}{2\phi}]^{sw}(f)$. Consequently, $|g(x)-f(x)|<\frac{1}{2\phi(x)}$ for all $x\in B^{2\delta_2}$ for some $\delta_2>0$. Find $C\in\varLambda$ such that $\overline{B}\subseteq C \subseteq B^{\delta_2}$. Clearly, $[C,\frac{1}{2\phi}]^{sw}(g) \in \mathscr{A}$. If $h\in [C,\frac{1}{2\phi}]^{sw}(g)$, then $|h(x)-g(x)|<\frac{1}{2\phi(x)}$ for all $x\in C^{\delta_3}$ for some $\delta_3>0$. Thus we have $|h(x)-f(x)|\leq|h(x)-g(x)|+|g(x)-f(x)| < \frac{1}{2\phi(x)} + \frac{1}{2\phi(x)} \leq \epsilon(x)$ for all $x\in B^\delta$ for $\delta=\min\{\delta_1,\delta_2,\delta_3\}$. Hence $f\in [C,\frac{1}{2\phi}]^{sw}(g) \subseteq [B,\epsilon]^{sw}(f)$. 
	
$(ii)$. It can be proved in a manner similar to part $(i)$.   
\end{proof}
	
%

The following theorem follows from Propositions \ref{dn_B and nw_B is less than nw of SWC} and \ref{nw of SWC is less than dn_B nw_B density of SWC}.

\begin{theorem}\label{network weight of topology of SWC} Let $\mathcal{B}$ be a bornology on a metric space $(X,d)$ with a closed base. Then
	\begin{enumerate}[(i)]
		\item $nw(C(X),\tau_{\mathcal{B}}^{sw}) = dn_{\mathcal{B}}^s(X)\cdot d(C(X),\tau_{\mathcal{B}}^{sw})\cdot nw(X,\mathcal{B})$.
		\item $nw(C(X),\tau_{\mathcal{B}}^{w}) = dn_{\mathcal{B}}(X)\cdot d(C(X),\tau_{\mathcal{B}}^{w})\cdot nw(X,\mathcal{B})$.
	\end{enumerate}
\end{theorem}

\subsection{Index of narrowness and Weight} \hfill \vspace{0.2cm}

In this subsection, we investigate the relationship of index of narrowness of the spaces $(C(X),\tau_{\mathcal{B}}^{sw})$ and $(C(X),\tau_{\mathcal{B}}^{w})$ with some other cardinal invariants.

\begin{proposition}\label{ib of SWT}Let $(X,d)$ be a metric space and let $\mathcal{B}$ be a bornology on $X$ with a closed base. Then the following inequalities hold:
	\begin{enumerate}[(i)]
		\item $dn_{\mathcal{B}}^s(X) \leq ib(C(X),\tau_{\mathcal{B}}^{sw})\cdot w_X(\mathcal{B})$.
		\item $dn_{\mathcal{B}}(X) \leq ib(C(X),\tau_{\mathcal{B}}^{w})\cdot w_X(\mathcal{B})$.
	\end{enumerate} 	 
\end{proposition}
\begin{proof} 
	
	
	$(i)$. Let $\mathcal{B}_0$ be a base of $\mathcal{B}$ with $|\mathcal{B}_0|=w_X(\mathcal{B})$. For each $B\in\mathcal{B}_0$, there is a subset $S_B$ of $C(X)$ such that $C(X)=[B,1]^{sw}(f_0) + S_B$ with $|S_B|\leq ib(C(X),\tau_{\mathcal{B}}^{sw})$. Consider $E=\cup_{B\in\mathcal{B}_0} E_B$, where $E_B=\{1+f_B : f_B \in S_B\}$ for each $B\in \mathcal{B}_0$. Clearly, $|E|\leq ib(C(X),\tau_{\mathcal{B}}^{sw})\cdot w_X(\mathcal{B})$. Let $g\in C(X)$ and $B_1\in\mathcal{B}$. So there exist a $B\in\mathcal{B}_0$ with $B_1\subseteq B$ and $f_B\in S_B$ such that $g=h+f_B$ for some $h\in [B,1]^{sw}(f_0)$. Thus we have $1+f_B \in E$ such that $g\leq 1+f_B$ for all $x\in B_1^\delta\subseteq B^\delta$ for some $\delta >0$. Therefore $E$ is a strongly dominating subset of $C(X)$ with respect to $\mathcal{B}$. Hence $dn_{\mathcal{B}}^s(X) \leq ib(C(X),\tau_{\mathcal{B}}^{sw})\cdot w_X(\mathcal{B})$.	
	
	$(ii)$. It can be proved in a manner similar to part $(i)$.  
\end{proof}

\begin{remark}\label{cardinals inequalities}For a topological group $X$, we have 
	\begin{enumerate}
		\item[(i)] $ib(X)\leq c(X)\leq d(X)\leq nw(X)\leq w(X)$ and $ib(X)\leq L(X)\leq nw(X)$ (see Proposition 5.2.1 of \cite{tkachenko} and 3.12.7 of \cite{engelking}). 
		\item[(ii)] $w(X)=\chi(X)\cdot ib(X)=\chi(X)\cdot d(X)=\chi(X)\cdot L(X)$ (see Theorems 5.2.3 and 5.2.5 in \cite{tkachenko}). 
	\end{enumerate}
\end{remark}

The following corollary to Proposition \ref{ib of SWT} follows immediately from the Remark \ref{cardinals inequalities}.
\begin{corollary} Let $(X,d)$ be a metric space and let $\mathcal{B}$ be a bornology on $X$ with a closed base. Then the following inequalities hold:
	\begin{enumerate}[(i)]
		\item $dn_{\mathcal{B}}^s(X) \leq c(C(X),\tau_{\mathcal{B}}^{sw})\cdot w_X(\mathcal{B})$; and $dn_{\mathcal{B}}(X) \leq c(C(X),\tau_{\mathcal{B}}^{w})\cdot w_X(\mathcal{B})$.
		\item $dn_{\mathcal{B}}^s(X) \leq L(C(X),\tau_{\mathcal{B}}^{sw})\cdot w_X(\mathcal{B})$; and $dn_{\mathcal{B}}(X) \leq L(C(X),\tau_{\mathcal{B}}^{w})\cdot w_X(\mathcal{B})$.
	\end{enumerate} 	 
\end{corollary}

\begin{corollary}\label{DHHM result2} \textnormal{(\cite{DHHM}, Lemma 2.8)} For a metric space $(X,d)$, $dn(X) \leq c(C(X),\tau^{w})$; and $dn(X) \leq L(C(X),\tau^{w})$.
\end{corollary}

\begin{theorem}\label{weight of tau_B^sw} Let $(X,d)$ be a metric space and let $\mathcal{B}$ be a bornology on $X$ with a closed base. Then 
	\begin{align*}
	&w(C(X),\tau_{\mathcal{B}}^{sw})=w_X(\mathcal{B})\cdot ib(C(X),\tau_{\mathcal{B}}^{sw})=w_X(\mathcal{B})\cdot c(C(X),\tau_{\mathcal{B}}^{sw})\\ &=w_X(\mathcal{B})\cdot L(C(X),\tau_{\mathcal{B}}^{sw})=w_X(\mathcal{B})\cdot d(C(X),\tau_{\mathcal{B}}^{sw})=w_X(\mathcal{B})\cdot nw(C(X),\tau_{\mathcal{B}}^{sw}).
	\end{align*}
\end{theorem}
\begin{proof} It follows from Remark \ref{cardinals inequalities}, Theorem \ref{character of topology of SWC} and Proposition \ref{ib of SWT}$(i)$.
\end{proof}

\begin{theorem}\label{weight of tau_B^w} Let $(X,d)$ be a metric space and let $\mathcal{B}$ be a bornology on $X$ with a closed base. Then 
	\begin{align*}
	&w(C(X),\tau_{\mathcal{B}}^{w})=w_X(\mathcal{B})\cdot ib(C(X),\tau_{\mathcal{B}}^{w})=w_X(\mathcal{B})\cdot c(C(X),\tau_{\mathcal{B}}^{w})\\ &=w_X(\mathcal{B})\cdot L(C(X),\tau_{\mathcal{B}}^{w})=w_X(\mathcal{B})\cdot d(C(X),\tau_{\mathcal{B}}^{w})=w_X(\mathcal{B})\cdot nw(C(X),\tau_{\mathcal{B}}^{w}).
	\end{align*}
\end{theorem}
\begin{proof} It follows from Remark \ref{cardinals inequalities}, Theorem \ref{character of topology of wc} and Proposition \ref{ib of SWT}$(ii)$.
\end{proof}

\begin{corollary}\label{DHHM result3} \textnormal{(\cite{DHHM}, Theorem 2.9)} For a metric space $(X,d)$, 
	\begin{align*}
	w(C(X),\tau^{w})&=ib(C(X),\tau^{w}) = c(C(X),\tau^{w})=L(C(X),\tau^{w})\\&=d(C(X),\tau^{w})=nw(C(X),\tau^{w}).
	\end{align*}
\end{corollary}

\begin{remark}\label{General remark}\normalfont{At this stage, we would like to mention that all results proved in this paper related to the space $(C(X),\tau_{\mathcal{B}}^{w})$ are true  for a Tychonoff space $X$ (instead of metric space $(X,d)$) with essentially the same proofs. In particular, Corollary \ref{DHHM result1}, Corollary \ref{DHHM result2}, and Corollary \ref{DHHM result3} are true for a Tychonoff space $X$. Therefore, one can observe that Corollaries \ref{DHHM result2} and \ref{DHHM result3} provide simpler proofs of the results originally proved in \cite{DHHM}.}\end{remark}
%

\section{Comparison with $(C(X),\tau_{\mathcal{B}}^s)$}

In the final section of the paper, we examine the relationships of the cardinal invariants of $(C(X),\tau_{\mathcal{B}}^{sw})$ with the corresponding cardinal invariants of $(C(X),\tau_{\mathcal{B}}^{s})$.  In particular, we show that the character (tightness) of $(C(X),\tau_{\mathcal{B}}^{sw})$ may be strictly greater than the character (tightness) of $(C(X),\tau_{\mathcal{B}}^{s})$ (Example \ref{example of character and tightness}). However, the global cardinal invariants such as density, weight, network weight, and index of narrowness of $(C(X),\tau_{\mathcal{B}}^{sw})$ are equal to the corresponding cardinal invariant of $(C(X),\tau_{\mathcal{B}}^{s})$.

\begin{theorem}\label{character and tightness comparison}Let $(X,d)$ be a metric space and let $\mathcal{B}$ be a bornology on $X$ with a closed base. Then the following statements hold:
	\begin{enumerate}[(i)]
		\item $\chi(C(X),\tau_{\mathcal{B}}^{sw}) = dn_{\mathcal{B}}^s(X)\cdot\chi(C(X),\tau_{\mathcal{B}}^{s})$; and \\ $\chi(C(X),\tau_{\mathcal{B}}^{w}) = dn_{\mathcal{B}}(X)\cdot\chi(C(X),\tau_{\mathcal{B}})$.
		\vspace*{.2 cm}
		\item $t(C(X),\tau_{\mathcal{B}}^{sw}) = dn_{\mathcal{B}}^s(X)\cdot t(C(X),\tau_{\mathcal{B}}^{s})$; and \\ $t(C(X),\tau_{\mathcal{B}}^{w}) = dn_{\mathcal{B}}(X)\cdot t(C(X),\tau_{\mathcal{B}})$.
	\end{enumerate} 	   
\end{theorem}
\begin{proof} $(i)$. By Theorem 3.2 of \cite{Cfbafs}, $\chi(C(X),\tau_{\mathcal{B}}^{s})= \chi(C(X),\tau_{\mathcal{B}}^{s})=w_X(\mathcal{B})$. So the statement follows from Theorems \ref{character of topology of SWC} and \ref{character of topology of wc}.
	
$(ii)$. By Theorem 3.5 of \cite{Cfbafs}, $t(C(X),\tau_{\mathcal{B}}^{s})=L^s(X, \mathcal{B})$ and $t(C(X),\tau_{\mathcal{B}})=L(X, \mathcal{B})$. Hence the statement follows from Theorem \ref{tightness of topology of SWC}.	
\end{proof}




\begin{example}\label{example of character and tightness} Let $X=\mathbb{R}$ with the usual metric. Consider the bornology $\mathcal{B}$ on $\mathbb{R}$ with base $\mathcal{B}_0=\{[a,\infty): a\in \mathbb{R}\}$. Clearly, $w_{\mathbb{R}}(\mathcal{B})=\aleph_0$. Consequently, $L^s(\mathbb{R},\mathcal{B})=L(\mathbb{R},\mathcal{B})=\aleph_0$. By Theorems 3.2 and 3.4 of \cite{Cfbafs}, we have $\chi(C(\mathbb{R}),\tau_{\mathcal{B}}^{s})=\chi(C(\mathbb{R}),\tau_{\mathcal{B}}^{})=t(C(\mathbb{R}),\tau_{\mathcal{B}}^{s})=t(C(\mathbb{R}),\tau_{\mathcal{B}}^{})=\aleph_0$. 
	
It is easy to check that $\mathcal{B}$ is shielded from closed sets and has no compact base. By Theorems \ref{bornology has compact base iff strong domination number is aleph} and \ref{B is shielded from closed sets, then dn_B^s=dn_B}, we have $dn_{\mathcal{B}}^s(X) = dn_{\mathcal{B}}(X) > \aleph_0$, and by Proposition \ref{B is shielded from closed sets, then L^s(X,B)=L(X,B)}, we have $L^s(X,\mathcal{B}) = L(X,\mathcal{B})$. Therefore $\chi(C(\mathbb{R}),\tau_{\mathcal{B}}^{sw})=\chi(C(\mathbb{R}),\tau_{\mathcal{B}}^{w})>\aleph_0$ and $t(C(\mathbb{R}),\tau_{\mathcal{B}}^{sw})=t(C(\mathbb{R}),\tau_{\mathcal{B}}^{w})>\aleph_0$.\qed  \end{example}

\begin{proposition}\label{dn_B^s is less than density of SUT}Let $\mathcal{B}$ be a bornology on a metric space $(X,d)$ with a closed base. Then the following inequalities hold:
	\begin{enumerate}[(i)]
		\item $dn_{\mathcal{B}}^s(X) \leq d(C(X),\tau_{\mathcal{B}}^s) \leq d(C(X),\tau_{\mathcal{B}}^{sw})$.
		\item $dn_{\mathcal{B}}(X) \leq d(C(X),\tau_{\mathcal{B}}) \leq d(C(X),\tau_{\mathcal{B}}^w)$.
	\end{enumerate} 	  
\end{proposition}
\begin{proof} $(i)$. Let $D\subseteq C(X)$ be dense in $(C(X),\tau_{\mathcal{B}}^{s})$ with $|D|=d(C(X),\tau_{\mathcal{B}}^{s})$. Let $F=\{1+|h|\in C^+(X): h\in D\}$. We show that $F$ is a strongly dominating subset of $C(X)$ with respect to $\mathcal{B}$. Let $f\in C(X)$ and $B\in \mathcal{B}$. So there is a function $h'\in D \cap [B,1]^s(f)$. Thus $f(x)< 1+|h(x)|$ for all $x\in B^\delta$ for some $\delta>0$. Therefore $dn_{\mathcal{B}}^s(X) \leq d(C(X),\tau_{\mathcal{B}}^s)$. Second inequality follows as  $\tau_{\mathcal{B}}^{sw}$ is finer than $\tau_{\mathcal{B}}^s$. 
	
$(ii)$. It can be proved in a manner similar to part $(i)$. 
\end{proof}	

\begin{proposition}\label{density of SWT is less than dn_B^s into density of SUT} Let $\mathcal{B}$ be a bornology on a metric space $(X,d)$ with a closed base. Then the following inequalities hold:
	\begin{enumerate}[(i)]
		\item $d(C(X),\tau_{\mathcal{B}}^{sw}) \leq dn_{\mathcal{B}}^s(X)\cdot d(C(X),\tau_{\mathcal{B}}^{s})$.
		\item $d(C(X),\tau_{\mathcal{B}}^{w}) \leq dn_{\mathcal{B}}(X)\cdot d(C(X),\tau_{\mathcal{B}})$.
	\end{enumerate} 	  
\end{proposition}
\begin{proof} $(i)$. Let $F_{\mathcal{B}}^s\subseteq C^+(X)$ be a strongly dominating subset of $C(X)$ with respect to $\mathcal{B}$ such that $|F_{\mathcal{B}}^s| = dn_{\mathcal{B}}^s(X)$, and let $D\subseteq C(X)$ be dense in $(C(X),\tau_{\mathcal{B}}^{s})$ with $|D|=d(C(X),\tau_{\mathcal{B}}^{s})$. Consider the set $\mathcal{S}=\left\lbrace \dfrac{g}{\phi}:g\in D \text{ and } \phi\in F_{\mathcal{B}}^s\right\rbrace \subseteq C(X)$. Clearly, $|\mathcal{S}|= dn_{\mathcal{B}}^s(X)\cdot d(C(X),\tau_{\mathcal{B}}^{s})$. It suffices to show that $\mathcal{S}$ is dense in $(C(X),\tau_{\mathcal{B}}^{sw})$. Let $[B,\epsilon]^{sw}(f)$ be a basic open set in $(C(X),\tau_{\mathcal{B}}^{sw})$ for some $f\in C(X)$. There exists $\phi\in F_{\mathcal{B}}^s$ such that $\frac{1}{\phi(x)}\leq \epsilon(x)$ for all $x\in B^{\delta_1}$ for some $\delta_1 >0$. Since $D$ is dense in $(C(X),\tau_{\mathcal{B}}^{s})$, there exists $g\in D\cap [B,1]^s(\phi f)$. This implies that $|g(x)-(\phi f)(x)|<1$ for all $x\in B^{\delta_2}$ for some $\delta_2 >0$. Consequently, $\left| \dfrac{g}{\phi(x)}-f(x)\right| <\dfrac{1}{\phi(x)}\leq \epsilon(x)$ for all $x\in B^{\delta}$ for $\delta=\min\{\delta_1,\delta_2\}$. Therefore $\dfrac{g}{\phi}\in \mathcal{S}\cap [B,\epsilon]^{sw}(f)$.   
	
	$(ii)$. It can be proved in a manner similar to part $(i)$. 
\end{proof}

The following theorem shows that the density of the space $(C(X),\tau_{\mathcal{B}}^{sw})$ is equal to the density of $(C(X),\tau_{\mathcal{B}}^{s})$.  

\begin{theorem}\label{density of SWT and WT}Let $\mathcal{B}$ be a bornology on a metric space $(X,d)$ with a closed base. Then $d(C(X),\tau_{\mathcal{B}}^s) = d(C(X),\tau_{\mathcal{B}}^{sw})$; and $d(C(X),\tau_{\mathcal{B}}) = d(C(X),\tau_{\mathcal{B}}^{w})$.
\end{theorem}
\begin{proof} It follows from Propositions \ref{dn_B^s is less than density of SUT} and \ref{density of SWT is less than dn_B^s into density of SUT}.
\end{proof} 

\begin{theorem}\label{w, nw, ib comparison}Let $(X,d)$ be a metric space and let $\mathcal{B}$ be a bornology on $X$ with a closed base. Then the following statements hold:
	\begin{enumerate}[(i)]
		\item $w(C(X),\tau_{\mathcal{B}}^{sw}) = w(C(X),\tau_{\mathcal{B}}^s)$; and $w(C(X),\tau_{\mathcal{B}}^{w}) = w(C(X),\tau_{\mathcal{B}})$.
		\item $nw(C(X),\tau_{\mathcal{B}}^{sw}) = nw(C(X),\tau_{\mathcal{B}}^s)$; and $nw(C(X),\tau_{\mathcal{B}}^{w}) = nw(C(X),\tau_{\mathcal{B}})$.
		\item $ib(C(X),\tau_{\mathcal{B}}^{sw}) = ib(C(X),\tau_{\mathcal{B}}^s)$; and $ib(C(X),\tau_{\mathcal{B}}^{w}) = ib(C(X),\tau_{\mathcal{B}})$.
	\end{enumerate} 	   
\end{theorem}
\begin{proof} We only give proofs for the statements involving $\tau_{\mathcal{B}}^{sw}$.
	
$(i)$. By Theorem \ref{weight of tau_B^sw}, $w(C(X),\tau_{\mathcal{B}}^{sw}) = w_X(\mathcal{B})\cdot d(C(X),\tau_{\mathcal{B}}^{sw})$. Also by Theorem 3.2 of \cite{Cfbafs} and Remark \ref{cardinals inequalities}(ii), $w(C(X),\tau_{\mathcal{B}}^{s}) = w_X(\mathcal{B})\cdot d(C(X),\tau_{\mathcal{B}}^{s})$. Hence the result follows from Theorem \ref{density of SWT and WT}.
	
$(ii)$. By Theorem \ref{network weight of topology of SWC}, $nw(C(X),\tau_{\mathcal{B}}^{sw}) = dn_{\mathcal{B}}^s(X)\cdot  d(C(X),\tau_{\mathcal{B}}^{sw})\cdot nw(X,\mathcal{B})$, and by Theorem 3.8 of \cite{Cfbafs}, we have $nw(C(X),\tau_{\mathcal{B}}^{s}) = d(C(X),\tau_{\mathcal{B}}^{s})\cdot nw(X,\mathcal{B})$. Using Proposition \ref{dn_B^s is less than density of SUT} and Theorem \ref{density of SWT and WT}, we get $nw(C(X),\tau_{\mathcal{B}}^{sw})\leq nw(C(X),\tau_{\mathcal{B}}^{s})$. The reverse inequality $nw(C(X),\tau_{\mathcal{B}}^{sw})\geq nw(C(X),\tau_{\mathcal{B}}^{s})$ follows as $\tau_{\mathcal{B}}^{s} \subseteq \tau_{\mathcal{B}}^{sw}$. 
	
$(iii)$. Since $\tau_{\mathcal{B}}^{sw}$ is finer than $\tau_{\mathcal{B}}^{s}$, $ib(C(X),\tau_{\mathcal{B}}^{s}) \leq ib(C(X),\tau_{\mathcal{B}}^{sw})$. For reverse inequality, let $[B,\epsilon]^{sw}(f_0)$ be a basic neighborhood of the identity $f_0$ in $(C(X),\tau_{\mathcal{B}}^{sw})$ for some $B\in\mathcal{B}$ and $\epsilon\in C^+(X)$. For the basic neighborhood $[B,1]^{s}(f_0)$ of the identity $f_0$ in $(C(X),\tau_{\mathcal{B}}^{s})$, there exists a subset $S$ of $C(X)$ such that $|S|\leq ib(C(X),\tau_{\mathcal{B}}^{s})$ and $C(X)=[B,1]^{s}(f_0)+S$. If $S'=\{\epsilon f : f\in S\}$, then $|S'| = |S| \leq  ib(C(X),\tau_{\mathcal{B}}^{s})$. We claim that $C(X)=[B,\epsilon]^{sw}(f_0)+S'$. Let $g\in C(X)$. So there exist $h\in [B,1]^{s}(f_0)$ and $f\in S$ such that $g/\epsilon=h+f$. Thus $g=\epsilon h+\epsilon f$, where $\epsilon h\in [B,\epsilon]^{sw}(f_0)$ and $\epsilon f \in S'$. Therefore $ib(C(X),\tau_{\mathcal{B}}^{sw}) \leq ib(C(X),\tau_{\mathcal{B}}^{s})$. 
\end{proof}

\bibliographystyle{plain}
\bibliography{reference_file_cardinal_functions}

\def\cprime{$'$} \def\cprime{$'$} \def\cprime{$'$}
\begin{thebibliography}{10}

\bibitem{tkachenko}
A.~Arhangel'ski{\u\i} and M.~Tkachenko.
\newblock {\em Topological groups and related structures}.
\newblock Atlantis Stud. Math. Vol. 1. Atlantis Press, Amsterdam, Paris, 2008.

\bibitem{Bcas}
G.~Beer, C.~Costantini, and S.~Levi.
\newblock Bornological convergence and shields.
\newblock {\em Mediterranean Journal of Mathematics}, 10(1):529--560, 2013.

\bibitem{Ballf}
G.~Beer and M.~I. Garrido.
\newblock Bornologies and locally {L}ipschitz functions.
\newblock {\em Bulletin of the Australian Mathematical Society},
  90(2):257--263, 2014.

\bibitem{Suc}
G.~Beer and S.~Levi.
\newblock Strong uniform continuity.
\newblock {\em Journal of Mathematical Analysis and Applications},
  350(2):568--589, 2009.

\bibitem{Tbab}
G.~Beer and S.~Levi.
\newblock Total boundedness and bornologies.
\newblock {\em Topology and Its Applications}, 156(7):1271--1288, 2009.

\bibitem{Ucucas}
G.~Beer and S.~Levi.
\newblock Uniform continuity, uniform convergence, and shields.
\newblock {\em Set-Valued and Variational Analysis}, 18(3-4):251--275, 2010.

\bibitem{SWc}
A.~Caserta.
\newblock Strong {W}hitney convergence.
\newblock {\em Filomat}, 26(1):81--91, 2012.

\bibitem{ATasucob}
A.~Caserta, G.~Di~Maio, and L.~Hol{\'a}.
\newblock Arzel{\`a}'s {T}heorem and strong uniform convergence on bornologies.
\newblock {\em Journal of Mathematical Analysis and Applications},
  371(1):384--392, 2010.

\bibitem{SWcob}
T.~K. Chauhan and V.~Jindal.
\newblock Strong {W}hitney convergence on bornologies.
\newblock {\em Submitted for Publication}, pages 1--15.

\bibitem{SWasucob}
T.~K. Chauhan and V.~Jindal.
\newblock Strong {W}hitney and strong uniform convergences on a bornology.
\newblock {\em Journal of Mathematical Analysis and Applications}, 505:125634,
  2022.

\bibitem{engelking}
R.~Engelking.
\newblock {\em General topology}.
\newblock Sigma series in pure mathematics, Vol. 6, Heldermann Verlag, Berlin,
  1989.

\bibitem{Hewitt}
E.~Hewitt.
\newblock Rings of real-valued continuous functions. {I}.
\newblock {\em Transactions of the American Mathematical Society}, 64:45--99,
  1948.

\bibitem{Cmotosucob}
L.~Hol{\'a}.
\newblock Complete metrizability of topologies of strong uniform convergence on
  bornologies.
\newblock {\em Journal of Mathematical Analysis and Applications},
  387(2):770--775, 2012.

\bibitem{Bpifs}
L.~Hol{\'a} and Lj. D.~R. Ko{\v{c}}inac.
\newblock Boundedness properties in function spaces.
\newblock {\em Quaestiones Mathematicae}, 41(6):829--838, 2018.

\bibitem{Cfbafs}
L.~Hol{\'a} and B.~Novotn\'{y}.
\newblock Cardinal functions, bornologies and function spaces.
\newblock {\em Annali di Matematica Pura ed Applicata (1923-)},
  193(5):1319--1327, 2014.

\bibitem{hola_novo2}
L.~Hol\'a and B.~Novotn\'y.
\newblock Topology of uniform convergence and $m$-topology on ${C(X)}$.
\newblock {\em Mediterranean Journal of Mathematics},
  14(2):doi.org/10.1007/s00009--017--0861--6, 2017.

\bibitem{mkatetov}
M.~Kat\v{e}tov.
\newblock Remarks on character and pseudocharacter.
\newblock {\em Commentationes Mathematicae Universitatis Carolinae}, 1:20--25,
  1960.

\bibitem{Bc}
A.~Lechicki, S.~Levi, and A.~Spakowski.
\newblock Bornological convergences.
\newblock {\em Journal of Mathematical Analysis and Applications},
  297(2):751--770, 2004.

\bibitem{DHHM}
G.~Di Maio, L.~Hol\'a, D.~Hol\'y, and R.~A. McCoy.
\newblock Topologies on the space of continuous functions.
\newblock {\em Topology and Its Applications}, 86:105--122, 1998.

\bibitem{M1}
R.~A. McCoy.
\newblock Fine topology on function spaces.
\newblock {\em International Journal of Mathematics and Mathematical Sciences},
  9:417--424, 1986.

\bibitem{Fswufagt}
R.~A. McCoy, S.~Kundu, and V.~Jindal.
\newblock {\em Function spaces with uniform, fine and graph topologies}.
\newblock Springer Briefs in Mathematics, Springer, Cham, 2018.

\bibitem{Tposocf}
R.~A. McCoy and I.~Ntantu.
\newblock {\em Topological properties of spaces of continuous functions},
  volume 1315.
\newblock springer, 2006.

\bibitem{whitney}
H.~Whitney.
\newblock Differentiable manifolds.
\newblock {\em Annals of Mathematics}, 37:645--680, 1936.

\end{thebibliography}

\end{document}